\documentclass{amsart}
\usepackage{amsfonts}

\newtheorem{theorem}{Theorem}
\theoremstyle{plain}

\newtheorem{corollary}{Corollary}

\numberwithin{equation}{section}

\begin{document}
\title[Some unified Integrals Associated with Bessel-Struve kernel
function]{Some unified Integrals Associated with Bessel-Struve kernel
function}

\author{K. S. Nisar}
\address{Department of Mathematics, College of Arts and Science-Wadi Addwasir\\
Prince Sattam bin Abdulaziz University, Saudi Arabia}
\email{ksnisar1@gmail.com}

\author{P. Agarwal}
\address{Department of Mathematics, Anand International College of Engineering, Jaipur 303012, Rajasthan, India.}

\author{S. Jain}
\address{Department of Mathematics, Poornima College of Engineering,Jaipur 302002,
Rajasthan, India.}


\begin{abstract}
In this paper, we discuss the generalized integral formula involving
Bessel-Struve kernel function $S_{\alpha }\left( \lambda z\right) $, which
expressed in terms of generalized Wright functions. Many interesting special
cases also obtained in this study.
\end{abstract}

\maketitle

\section{Introduction}

In 1888 Pincherle studied the integrals involving product of Gamma functions
along vertical lines (see \cite{Pincherle,Pincherle1,Pincherle2}). Latterly,
Barnes \cite{Barnes} , Mellin \cite{Mellin} and Cahen \cite{Cahen} extended
the study and applied some of these integrals in the study of Riemann zeta
function and other Drichlet's series. The integral formulas involving
special functions have been developed by many researchers (\cite{Brychkov},%
\cite{Choi1}). In \cite{Garg} presented unified integral representation of
Fox H-functions and in \cite{Ali} hypergeometric $_{2}F_{1}$ functions.
Recently J. Choi and P. Agarwal \cite{Choi2} obtained two unified integral
representations of Bessel functions $J_{v}\left( z\right) $. Also, many
interesting integral formula involving $J_{v}\left( z\right) $ is given in 
\cite{Brychkov}and \cite{Watson}.

The Bessel-Struve kernel $S_{\alpha }\left( \lambda z\right) ,\lambda \in 
\mathbb{C},$ \cite{Gasmi} which is unique solution of the initial value
problem $l_{\alpha }u\left( z\right) =\lambda ^{2}u\left( z\right) $ with
the initial conditions $u\left( 0\right) =1$ and $u^{^{\prime }}\left(
0\right) =\lambda \Gamma \left( \alpha +1\right) /\sqrt{\pi }\Gamma \left(
\alpha +3/2\right) $ is given by 
\begin{equation*}
S_{\alpha }\left( \lambda z\right) =j_{\alpha }\left( i\lambda z\right)
-ih_{\alpha }\left( i\lambda z\right) ,\forall z\in C
\end{equation*}

where $j_{\alpha }$ and $h_{\alpha }$ are the normalized Bessel and Struve
functions

Moreover, the Bessel-Struve kernel is a holomorphic function on $\mathbb{C}%
\times \mathbb{C}$ and it can be expanded in a power series in the form

\begin{equation}
S_{\alpha }\left( \lambda z\right) =\sum_{n=0}^{\infty }\frac{\left( \lambda
z\right) ^{n}\Gamma \left( \alpha +1\right) \Gamma \left( \left( n+1\right)
/2\right) }{\sqrt{\pi }n!\Gamma \left( n/2+\alpha +1\right) },
\label{Bessel-Struve}
\end{equation}

The generalized Wright hypergeometric function ${}_{p}\psi _{q}(z)$ is given
by the series 
\begin{equation}
{}_{p}\psi _{q}(z)={}_{p}\psi _{q}\left[ 
\begin{array}{c}
(a_{i},\alpha _{i})_{1,p} \\ 
(b_{j},\beta _{j})_{1,q}%
\end{array}%
\bigg|z\right] =\displaystyle\sum_{k=0}^{\infty }\dfrac{\prod_{i=1}^{p}%
\Gamma (a_{i}+\alpha _{i}k)}{\prod_{j=1}^{q}\Gamma (b_{j}+\beta _{j}k)}%
\dfrac{z^{k}}{k!},  \label{eqn-4-Struve}
\end{equation}%
where $a_{i},b_{j}\in \mathbb{C}$, and real $\alpha _{i},\beta _{j}\in 
\mathbb{R}$ ($i=1,2,\ldots ,p;j=1,2,\ldots ,q$). Asymptotic behavior of this
function for large values of argument of $z\in {\mathbb{C}}$ were studied in 
\cite{Foxc} and under the condition 
\begin{equation}
\displaystyle\sum_{j=1}^{q}\beta _{j}-\displaystyle\sum_{i=1}^{p}\alpha
_{i}>-1  \label{eqn-5-Struve}
\end{equation}%
was found in the work of \cite{Wright-2,Wright-3}. Properties of this
generalized Wright function were investigated in \cite{Kilbas}, (see also 
\cite{Kilbas-itsf, Kilbas-frac}. In particular, it was proved \cite{Kilbas}
that ${}_{p}\psi _{q}(z)$, $z\in {\mathbb{C}}$ is an entire function under
the condition ($\ref{eqn-5-Struve}$).

The generalized hypergeometric function represented as follows \cite%
{Rainville}:

\begin{equation}
_{p}F_{q}\left[ 
\begin{array}{c}
\left( \alpha _{p}\right) ; \\ 
\left( \beta _{q}\right) ;%
\end{array}%
z\right] =\sum\limits_{n=0}^{\infty }\frac{\Pi _{j=1}^{p}\left( \alpha
_{j}\right) _{n}}{\Pi _{j=1}^{p}\left( \beta _{j}\right) _{n}}\frac{z^{n}}{n!%
},  \label{eqn-1-hyper}
\end{equation}

provided $p\leq q; p=q+1$ and $\left\vert z\right\vert <1$

where $\left( \lambda \right) _{n}$ is well known Pochhammer symbol defined
for $\left( \text{ for }\lambda \in C\right) $ (see \cite{Rainville})

\begin{equation}
\left( \lambda \right) _{n}:=\left\{ 
\begin{array}{c}
1\text{ \ \ \ \ \ \ \ \ \ \ \ \ \ \ \ \ \ \ \ \ \ \ \ \ \ \ \ \ \ \ \ \ \ \
\ \ }\left( n=0\right) \\ 
\lambda \left( \lambda +1\right) ....\left( \lambda +n-1\right) \text{ \ \ \
\ \ \ \ \ \ \ \ \ \ }\left( n\in N:=\{1,2,3....\}\right)%
\end{array}%
\right.  \label{eqn-2-hyper}
\end{equation}

\begin{equation}
\left( \lambda \right) _{n}=\frac{\Gamma \left( \lambda +n\right) }{\Gamma
\left( \lambda \right) }\text{ \ \ \ \ \ \ \ \ \ \ }\left( \lambda \in
C\backslash Z_{0}^{-}\right) .  \label{eqn-2b-hyper}
\end{equation}

where $Z_{0}^{-}$ is the set of nonpositive integers.

If we put $\alpha _{1}=...=\alpha _{p}=\beta _{1}=....=\beta _{q} $ in ($\ref%
{eqn-4-Struve}$),then ($\ref{eqn-1-hyper}$) is a special case of the
generalized Wright function:

\begin{equation}
{}_{p}\psi _{q}(z)={}_{p}\psi _{q}\left[ 
\begin{array}{c}
\left( \alpha _{1},1\right) ,...,\left( \alpha _{p},1\right) ; \\ 
\left( \beta _{1},1\right) ,...,\left( \beta _{q},1\right) ;%
\end{array}%
z\right] =\dfrac{\prod_{j=1}^{p}\Gamma (\alpha _{j})}{\prod_{j=1}^{q}\Gamma
(\beta _{j})}\text{ }_{p}F_{q}\left[ 
\begin{array}{c}
\alpha _{1},...,\alpha _{p}; \\ 
\beta _{1},...,\beta _{q};%
\end{array}%
z\right]  \label{eqn-3-hyper}
\end{equation}

For the present investigation, we need the following result of Oberbettinger 
\cite{Ober}

\begin{equation}
\int_{0}^{\infty }x^{\mu -1}\left( x+a+\sqrt{x^{2}+2ax}\right) ^{-\lambda
}dx=2\lambda a^{-\lambda }\left( \frac{a}{2}\right) ^{\mu }\frac{\Gamma
\left( 2\mu \right) \Gamma \left( \lambda -\mu \right) }{\Gamma \left(
1+\lambda +\mu \right) }  \label{eqn-int1}
\end{equation}

provided $0<Re\left( \mu \right) <Re\left( \lambda \right) $

Motivated by the work of \cite{Choi1} , here we present the integral
formulas of Bessel-Struve Kernel function of first kind $S_{\alpha }\left(
\lambda z\right) $ ,$\lambda \in \mathbb{C},$which expressed interns of
generalized Wright or generalized hypergeometric functions. 

\section{Main results}

Two generalized integral formulas established here, which expressed in terms
of generalized (Wright) hypergeometric functions $\left( \ref{eqn-3-hyper}%
\right) $ by inserting the Bessel-Struve kernel function of the first kind $%
\left( \ref{Bessel-Struve}\right) $ with the suitable argument in the
integrand of $\left( \ref{eqn-int1}\right) $

\begin{theorem}
For $\lambda ,\mu ,\nu ,\gamma \in \mathbb{C},$and $x>0$, $\mathbb{R}\left(
\lambda \right) >\mathbb{R}\left( \mu \right) >0,$ then the  following
integral formula holds true:

\begin{eqnarray}
&&\int_{0}^{\infty }x^{\mu -1}\left( x+a+\sqrt{x^{2}+2ax}\right) ^{-\lambda
}S_{\alpha }\left( \frac{\gamma y}{x+a+\sqrt{x^{2}+2ax}}\right) dx  \notag \\
&=&\frac{2^{1-\mu }a^{\mu -\lambda }\Gamma \left( \alpha +1\right) \Gamma
\left( 2\mu \right) }{\sqrt{\pi }}  \notag \\
&&\times _{3}\Psi _{2}\left[ 
\begin{array}{c}
\left( \frac{1}{2},\frac{1}{2}\right) ,\left( \lambda +1,1\right) ,\left(
\lambda -\mu ,1\right) ; \\ 
\left( \lambda ,1\right) ,\left( 1+\lambda +\mu ,1\right) ;%
\end{array}%
\frac{\gamma y}{a}\right]  \label{1}
\end{eqnarray}
\end{theorem}

\begin{proof}
Consider the series representation of $S_{\alpha }\left( \frac{\gamma y}{x+a+%
\sqrt{x^{2}+2ax}}\right) $ and applying $\left( \ref{eqn-int1}\right) .$ By
interchanging the order of integration and summation,which verified by
uniform convergence of the involved series under the given conditions, we get

\begin{eqnarray*}
&&\int_{0}^{\infty }x^{\mu -1}\left( x+a+\sqrt{x^{2}+2ax}\right) ^{-\lambda
}S_{\alpha }\left( \frac{\gamma y}{x+a+\sqrt{x^{2}+2ax}}\right) dx \\
&=&\int_{0}^{\infty }x^{\mu -1}\left( x+a+\sqrt{x^{2}+2ax}\right) ^{-\lambda
}\sum_{n=0}^{\infty }\frac{\left( \frac{\gamma y}{x+a+\sqrt{x^{2}+2ax}}%
\right) ^{n}\Gamma \left( \alpha +1\right) \Gamma \left( \frac{n+1}{2}%
\right) }{\sqrt{\pi }\Gamma \left( \frac{n}{2}+\alpha +1\right) n!}dx \\
&=&\sum_{n=0}^{\infty }\frac{\left( \gamma y\right) ^{n}\Gamma \left( \alpha
+1\right) \Gamma \left( \frac{n+1}{2}\right) }{\sqrt{\pi }\Gamma \left( 
\frac{n}{2}+\alpha +1\right) n!}\int_{0}^{\infty }x^{\mu -1}\left( x+a+\sqrt{%
x^{2}+2ax}\right) ^{-\left( \lambda +n\right) }dx
\end{eqnarray*}

in view of the conditions give in Theorem 1and  applying the integral
formula $\left( \ref{eqn-int1}\right) $ ,we obtain the following integral
representation:

\begin{eqnarray*}
&&\int_{0}^{\infty }x^{\mu -1}\left( x+a+\sqrt{x^{2}+2ax}\right) ^{-\lambda
}S_{\alpha }\left( \frac{\gamma y}{x+a+\sqrt{x^{2}+2ax}}\right) dx \\
&=&2^{1-\mu }a^{-\lambda +\mu }\Gamma \left( \alpha +1\right) \Gamma \left(
2\mu \right) \pi ^{-1/2}\sum_{n=0}^{\infty }\Gamma \left( \frac{n+1}{2}%
\right) \frac{\Gamma \left( \lambda +n+1\right) }{\Gamma \left( \lambda
+n\right) } \\
&&\times \frac{\Gamma \left( \lambda +n-\mu \right) }{\Gamma \left( \lambda
+\mu +n+1\right) }\frac{\gamma ^{n}y^{n}}{n!a^{n}}
\end{eqnarray*}

which,upon using $\left( \ref{eqn-3-hyper}\right) $ ,yeilds $\left( \ref{1}%
\right) $. This completes proof of theorem 1
\end{proof}

\begin{theorem}
For $\lambda ,\mu ,\nu ,\gamma \in \mathbb{C}$ with $0<\mathbb{R}\left( \mu
\right) <\mathbb{R}\left( \lambda +\nu \right) $ and $x>0$. The following
integral formula hold true:%
\begin{eqnarray*}
&&\int_{0}^{\infty }x^{\mu -1}\left( x+a+\sqrt{x^{2}+2ax}\right) ^{-\lambda
}S_{\alpha }\left( \frac{\gamma xy}{x+a+\sqrt{x^{2}+2ax}}\right) dx \\
&=&\frac{2^{1+\mu }a^{\mu -\lambda }\Gamma \left( \alpha +1\right) \Gamma
\left( \lambda -\mu \right) }{\sqrt{\pi }\Gamma \left( 1+\lambda +\mu
\right) }\text{ }_{3}\Psi _{2}\left[ 
\begin{array}{c}
\left( \frac{1}{2},\frac{1}{2}\right) ,\left( 2\mu ,2\right) ,\left( \lambda
+1,1\right) ; \\ 
\left( \lambda ,1\right) ,\left( \alpha +1,\frac{1}{2}\right) ;%
\end{array}%
\right] \gamma y
\end{eqnarray*}
\end{theorem}

\begin{proof}
Interchanging the order of integration and summation and the series
representation of Bessel Struve kerenel function, we get%
\begin{eqnarray*}
&&\int_{0}^{\infty }x^{\mu -1}\left( x+a+\sqrt{x^{2}+2ax}\right) ^{-\lambda
}S_{\alpha }\left( \frac{\gamma xy}{x+a+\sqrt{x^{2}+2ax}}\right) dx \\
&=&\int_{0}^{\infty }x^{\mu -1}\left( x+a+\sqrt{x^{2}+2ax}\right) ^{-\lambda
} \\
&&\times \sum_{n=0}^{\infty }\left( \frac{\gamma xy}{x+a+\sqrt{x^{2}+2ax}}%
\right) ^{n}\frac{\Gamma \left( \alpha +1\right) \Gamma \left( \frac{n+1}{2}%
\right) }{\sqrt{\pi }\Gamma \left( \frac{n}{2}+\alpha +1\right) n!}dx \\
&=&\sum_{n=0}^{\infty }\frac{\gamma ^{n}y^{n}\Gamma \left( \alpha +1\right)
\Gamma \left( \frac{n+1}{2}\right) }{\sqrt{\pi }\Gamma \left( \frac{n}{2}%
+\alpha +1\right) }\int_{0}^{\infty }x^{\mu +n-1}\left( x+a+\sqrt{x^{2}+2ax}%
\right) ^{-\left( \lambda +n\right) }dx
\end{eqnarray*}

in view of the condition give in theorem 1,we can apply the integral formula 
$\left( \ref{eqn-int1}\right) $ and obtain the following integral
representation:

\begin{eqnarray*}
&&\int_{0}^{\infty }x^{\mu -1}\left( x+a+\sqrt{x^{2}+2ax}\right) ^{-\lambda
}S_{\alpha }\left( \frac{\gamma xy}{x+a+\sqrt{x^{2}+2ax}}\right) dx \\
&=&\frac{2^{1+\mu }a^{\mu -\lambda }\Gamma \left( \alpha +1\right) \Gamma
\left( \lambda -\mu \right) }{\sqrt{\pi }\Gamma \left( 1+\lambda +\mu
\right) }\sum_{n=0}^{\infty }\Gamma \left( \frac{n+1}{2}\right) \frac{\Gamma
\left( \lambda +n+1\right) \Gamma \left( 2\mu +2n\right) }{\Gamma \left(
\lambda +n\right) \Gamma \left( \frac{n}{2}+\alpha +1\right) }\frac{\gamma
^{n}y^{n}}{n!a^{n}}
\end{eqnarray*}

which gives the desired result.
\end{proof}

\subsection{Representation of Bessel Struve kernel function in terms of
exponential function}

In this subsection we represent the Bessel Struve function in terms of
exponential function. Also, we derive the Marichev Saigo Maeda operator
representation of special cases. The representation Bessel Struve Kernel
function interms of exponential function as:%
\begin{equation}
S_{\frac{-1}{2}}\left( x\right) =e^{x},  \label{e1}
\end{equation}

\begin{equation}
S_{\frac{1}{2}}\left( x\right) =\frac{-1+e^{x}}{x}.  \label{e2}
\end{equation}

Now, we give the the following corollaries:

\begin{corollary}
For $\lambda ,\mu \in \mathbb{C}$ with $0<\mathbb{R}\left( \mu \right) <%
\mathbb{R}\left( \lambda \right) $ and $x>0$.The following integral formula
holds true%
\begin{eqnarray}
&&\int_{0}^{\infty }x^{\mu -1}\left( x+a+\sqrt{x^{2}+2ax}\right) ^{-\lambda
}e^{\left( \frac{y}{x+a+\sqrt{x^{2}+2ax}}\right) }dx  \label{3} \\
&=&2^{1-\mu }a^{\mu -\lambda }\Gamma \left( 2\mu \right) \text{ }_{2}\Psi
_{2}\left[ 
\begin{array}{c}
\left( \lambda +1,1\right) ,\left( \lambda -\mu ,1\right) ; \\ 
\left( \lambda ,1\right) ,\left( 1+\lambda -\mu \right) ;%
\end{array}%
\frac{y}{a}\right]   \notag
\end{eqnarray}

\begin{proof}
As same as in theorem 1 and theorem 2 , using the formula $\left( \ref%
{eqn-int1}\right) $ and $\left( \ref{e1}\right) $, one can easily reach the
result
\end{proof}
\end{corollary}

\begin{corollary}
Let the conditions given in Corollary 1 satisfied. Then the following
integral formula holds true%
\begin{eqnarray*}
&&\int_{0}^{\infty }x^{\mu -1}\left( x+a+\sqrt{x^{2}+2ax}\right) ^{-\lambda
}e^{\left( \frac{y}{x+a+\sqrt{x^{2}+2ax}}\right) }dx \\
&=&=\frac{2^{1-\mu }a^{\mu -\lambda }\Gamma \left( 2\mu \right) \Gamma
\left( \lambda +1\right) \Gamma \left( \lambda -\mu \right) \text{ }}{\Gamma
\left( \lambda \right) \Gamma \left( 1+\lambda -\mu \right) }\text{ }%
_{2}F_{2}\left[ 
\begin{array}{c}
\lambda +1,\lambda -\mu ; \\ 
\lambda ,1+\lambda -\mu ;%
\end{array}%
\frac{y}{a}\right] 
\end{eqnarray*}
\end{corollary}

\begin{proof}
In the view of equations $\left( \ref{eqn-1-hyper}\right) $ , $\left( \ref%
{eqn-2-hyper}\right) $ and $\left( \ref{3}\right) $ , we obtain the required
result.
\end{proof}

\begin{corollary}
For $\lambda ,\mu ,\in \mathbb{C}$ with $0<\mathbb{R}\left( \mu \right) <%
\mathbb{R}\left( \lambda \right) $ and $x>0$.The following integral formula
holds true%
\begin{eqnarray*}
&&\int_{0}^{\infty }x^{\mu -1}\left( x+a+\sqrt{x^{2}+2ax}\right) ^{-\lambda
}e^{\left( \frac{y}{x+a+\sqrt{x^{2}+2ax}}-1\right) }dx \\
&=&2^{-\mu }a^{\mu -\lambda }\Gamma \left( 2\mu \right) \text{ }_{3}\Psi _{3}%
\left[ 
\begin{array}{c}
\left( \frac{1}{2},\frac{1}{2}\right) ,\left( \lambda +1,1\right) ,\left(
\lambda -\mu ,1\right) ; \\ 
\left( \frac{1}{2},\frac{3}{2}\right) ,\left( \lambda ,1\right) ,\left(
1+\lambda +\mu ,1\right) ;%
\end{array}%
\frac{y}{a}\right] 
\end{eqnarray*}
\end{corollary}

\subsection{Relation between Bessel Struve kernel function and Bessel and
Struve function of \ first kind}

In this subsection we show the relation between $S_{\alpha }\left( x\right) $
and Bessel function $I_{v}\left( x\right) $ and Struve function $L_{v}\left(
x\right) $ by choosing particular values of $\alpha $

\begin{equation}
S_{0}\left( x\right) =I_{0}\left( x\right) +L_{0}\left( x\right) ,
\label{eqn-r1}
\end{equation}

\begin{equation}
S_{1}\left( x\right) =\frac{2I_{1}\left( x\right) +L_{1}\left( x\right) }{x},
\label{eqn-r2}
\end{equation}

In the light of above relations ,we have the following theorems:

\begin{theorem}
For $\lambda ,\mu \in \mathbb{C}$ with $0<\mathbb{R}\left( \mu \right) <%
\mathbb{R}\left( \lambda \right) $ and $x>0$.Then the following integral
formula holds true:%
\begin{eqnarray*}
&&\int_{0}^{\infty }x^{\mu -1}\left( x+a+\sqrt{x^{2}+2ax}\right) ^{-\lambda }
\\
&&\times \left[ I_{0}\left( \frac{y}{x+a+\sqrt{x^{2}+2ax}}\right)
+L_{0}\left( \frac{y}{x+a+\sqrt{x^{2}+2ax}}\right) \right] dx \\
&=&2^{1-\mu }a^{\mu -\lambda }\pi ^{-1/2}\Gamma \left( 2\mu \right) \text{ }%
_{3}\Psi _{3}\left[ 
\begin{array}{c}
\left( \frac{1}{2},\frac{1}{2}\right) ,\left( \lambda +1,1\right) ,\left(
\lambda -\mu ,1\right) ; \\ 
\left( 1,\frac{1}{2}\right) ,\left( \lambda ,1\right) ,\left( 1+\lambda +\mu
,1\right) ;%
\end{array}%
\frac{y}{a}\right] 
\end{eqnarray*}
\end{theorem}

\begin{proof}
Consider the relation given in $\left( \ref{eqn-r1}\right) $ and applying $%
\left( \ref{eqn-int1}\right) .$ By interchanging the order of integration
and summation,which verified by uniform convergence of the involved series
under the given conditions, we get%
\begin{eqnarray*}
&&\int_{0}^{\infty }x^{\mu -1}\left( x+a+\sqrt{x^{2}+2ax}\right) ^{-\lambda }
\\
&&\times \left[ I_{0}\left( \frac{y}{x+a+\sqrt{x^{2}+2ax}}\right)
+L_{0}\left( \frac{y}{x+a+\sqrt{x^{2}+2ax}}\right) \right] dx \\
&=&\int_{0}^{\infty }x^{\mu -1}\left( x+a+\sqrt{x^{2}+2ax}\right) ^{-\lambda
}\sum_{n=0}^{\infty }\left( \frac{y}{x+a+\sqrt{x^{2}+2ax}}\right) ^{n}\frac{%
\Gamma \left( \frac{n+1}{2}\right) }{\sqrt{\pi }n!\Gamma \left( \frac{n}{2}%
+1\right) }dx \\
&=&\sum_{n=0}^{\infty }\frac{\Gamma \left( \frac{n+1}{2}\right) y^{n}}{\sqrt{%
\pi }n!\Gamma \left( \frac{n}{2}+1\right) }\int_{0}^{\infty }x^{\mu
-1}\left( x+a+\sqrt{x^{2}+2ax}\right) ^{-\left( \lambda +n\right) }dx \\
&=&2^{1-\mu }a^{\mu -\lambda }\pi ^{-1/2}\Gamma \left( 2\mu \right)
\sum_{n=0}^{\infty }\frac{\Gamma \left( \frac{n+1}{2}\right) y^{n}}{n!\Gamma
\left( \frac{n}{2}+1\right) a^{n}}\frac{\Gamma \left( \lambda +n+1\right) }{%
\Gamma \left( \lambda +n\right) }\frac{\Gamma \left( \lambda +n-\mu \right) 
}{\Gamma \left( \lambda +n+1+\mu \right) }
\end{eqnarray*}

which gives the desired result.
\end{proof}

\begin{theorem}
The following integral formula holds true with  $\lambda ,\mu ,\in \mathbb{C}
$ with $0<\mathbb{R}\left( \mu \right) <\mathbb{R}\left( \lambda \right) $
and $x>0$.%
\begin{eqnarray*}
&&\int_{0}^{\infty }x^{\mu -1}\left( x+a+\sqrt{x^{2}+2ax}\right) ^{-\lambda }
\\
&&\times \left[ 2I_{1}\left( \frac{y}{x+a+\sqrt{x^{2}+2ax}}\right)
+L_{1}\left( \frac{y}{x+a+\sqrt{x^{2}+2ax}}\right) \right] dx \\
&=&2^{1-\mu }a^{\mu -\lambda }\pi ^{-1/2}\Gamma \left( 2\mu \right) \text{ }%
_{2}\Psi _{2}\left[ 
\begin{array}{c}
\left( \frac{1}{2},\frac{1}{2}\right) ,\left( \lambda -\mu ,1\right) ; \\ 
\left( 2,\frac{1}{2}\right) ,\left( 1+\lambda +\mu ,1\right) ;%
\end{array}%
\frac{y}{a}\right] 
\end{eqnarray*}
\end{theorem}

\begin{proof}
Consider the series representation of $S_{1}\left( \frac{y}{x+a+\sqrt{%
x^{2}+2ax}}\right) $ and applying $\left( \ref{eqn-int1}\right) $and
interchanging the order of integration and summation, we get%
\begin{eqnarray*}
&&\int_{0}^{\infty }x^{\mu -1}\left( x+a+\sqrt{x^{2}+2ax}\right) ^{-\lambda }
\\
&&\times \left[ 2I_{1}\left( \frac{y}{x+a+\sqrt{x^{2}+2ax}}\right)
+L_{1}\left( \frac{y}{x+a+\sqrt{x^{2}+2ax}}\right) \right] dx \\
&=&\int_{0}^{\infty }x^{\mu -1}\left( x+a+\sqrt{x^{2}+2ax}\right) ^{-\lambda
}S_{1}\left( \frac{y}{x+a+\sqrt{x^{2}+2ax}}\right) dx \\
&=&\sum_{n=0}^{\infty }\frac{\Gamma \left( \frac{n+1}{2}\right) y^{n}}{\sqrt{%
\pi }n!\Gamma \left( \frac{n}{2}+2\right) }\int_{0}^{\infty }x^{\mu
-1}\left( x+a+\sqrt{x^{2}+2ax}\right) ^{-\left( \lambda +n\right) }dx \\
&=&2^{1-\mu }a^{\mu -\lambda }\pi ^{-1/2}\Gamma \left( 2\mu \right)
\sum_{n=0}^{\infty }\frac{\Gamma \left( \frac{n+1}{2}\right) \Gamma \left(
\lambda +n-\mu \right) }{\Gamma \left( \frac{n}{2}+2\right) \Gamma \left(
1+\lambda +n+\mu \right) }\frac{y^{n}}{a^{n}n!}
\end{eqnarray*}

which gives the required result.
\end{proof}

\textbf{Conclusion}

The generalized integral formula involving Bessel-Struve kernel function $%
S_{\alpha }\left( \lambda z\right) $, which expressed in terms of
generalized Wright functions are given in this paper. Also the relation
between exponential function, Bessel function and Struve function with
Bessel-Struve kernel function is also discussed with particular cases.

\end{document}